\documentclass[12pt]{article}
\usepackage{amsmath}
\usepackage{amssymb}
\usepackage{amsmath,amsthm,amscd,amsfonts,amssymb,mathrsfs}
\usepackage[usenames]{color}

\usepackage{cite}
\usepackage{titlesec}

\allowdisplaybreaks

\numberwithin{equation}{section}
\titleformat*{\section}{\large\bfseries}

\newcommand{\RR}{\mathbb R}

\newcommand{\ZZ}{\mathbb Z}

\newtheorem{thm}{Theorem}[section]
\newtheorem{lemma}[thm]{Lemma}

\newtheorem{rem}[thm]{Remark}
\newtheorem{hyp}[thm]{Hypothesis}

\newcommand{\dprod}[2]{\left\langle #1,#2\right\rangle}
\newcommand{\norm}[1]{\left\lVert #1\right\rVert}

\title{Schr\"{o}dinger operators with decaying randomness - Pure point spectrum.}
\author{Anish Mallick\footnote{E-mail:\texttt{anish.mallick@icts.res.in}, Institute: ICTS-TIFR Bangaluru, India.} 
~\& Dhriti Ranjan Dolai\footnote{E-mail:\texttt{dhriti\_vs@isibang.ac.in}, Institute: ISI Bangalore, India.}}
\date{\today}

\begin{document}

\maketitle

\begin{abstract}
\noindent For Schr\"{o}dinger operator with decaying random potential
with fat tail single site distribution the negative spectrum shows a transition from essential spectrum to discrete spectrum.
We study the Schr\"{o}dinger operator 
$H^\omega=-\Delta+\displaystyle\sum_{n\in\mathbb{Z}^d}a_n\omega_n\chi_{_{(0,1]^d}}(x-n)$ on $L^2(\mathbb{R}^d)$. 
Here we take $a_n=O(|n|^{-\alpha})$ for large $n$ where $\alpha>0$, and $\{\omega_n\}_{n\in\mathbb{Z}^d}$ are i.i.d real random variables with absolutely continuous distribution $\mu$ 
such that $\frac{d\mu}{dx}(x)=O\big(|x|^{-(1+\delta)}\big)~as~|x|\to\infty$, for some $\delta>0$.
We show that $H^\omega$ exhibits exponential localization on negative part of spectrum independent of the parameters chosen.
For $\alpha\delta\leq d$ we show that the spectrum is entire real line almost surely,
but for $\alpha\delta>d$ we have $\sigma_{ess}(H^\omega)=[0,\infty)$ and negative part of the spectrum is discrete almost surely.
In some cases we show the existence of the absolutely continuous spectrum.
\end{abstract}
{\bf AMS 2010 MSC:} 81Q10, 35J10, 82B44, 47B80.

\begin{section}{Introduction}
We consider the random Schr\"{o}dinger operator $H^\omega$ on $L^2(\mathbb{R}^d)$ given by
\begin{equation}
\label{model}
H^\omega=-\Delta + V^\omega,
\end{equation}
where $\Delta$ is the Laplacian defined by
$$(\Delta f)(x)=\sum_{i=1}^d \frac{\partial^2 f}{\partial x_i^2}(x)\qquad\forall x\in\RR^d,f\in C_c(\RR^d),$$
is a self-adjoint operator and for identically distributed independent random variables $\omega:=\{\omega_n\}_{n\in\ZZ^d}$ the operator $V^\omega$ is defined by
\begin{align*}
(V^\omega f)(x)=\sum_{n\in\mathbb{Z}^d}a_n\omega_n \chi_{_{(0,1]^d}}(x-n)f(x)\qquad\forall x\in\RR^d,f\in C_c(\RR^d),
\end{align*}
where $\chi_{(0,1]^d}$ is characteristic function on the cube $(0,1]^d$. 
The sequence $a_n$ is non-negative and decays to zero at infinity at a rate $\alpha$, i.e $a_n= O(|n|^{-\alpha})$ for $n$ large.
The random variable $\omega_0$ follow a real absolutely continuous distribution $\mu$ with bounded density such that
\begin{equation}\label{probdisteq1}
 \frac{d\mu}{dx}(x)=O\big(|x|^{-(1+\delta)}\big)~as~|x|\to\infty,~~\delta>0.
\end{equation}
Following Kolmogorov construction, we will view $\omega$ as a random variable on the probability space $(\Omega,\mathcal{B},\mathbb{P})$.
Note that we can write the action of $V^\omega$ as multiplication by a random function, which will also denote as $V^\omega$, i.e
$$(V^\omega u)(x)=V^\omega(x) u(x)\qquad \forall x\in\RR^d,u\in C_c(\mathbb{R}^d).$$
By above definitions it is easy to see that $C_c^\infty(\RR^d)$ is domain of definition for $H^\omega$ almost surely, hence the operator $H^\omega$ is densely defined. 
Though it should be noted that even if $V^\omega$ is bounded in any compact set, this does not imply that $H^\omega$ is essentially self adjoint. 
We will focus our work in the region of $\alpha$ and $\delta$, where $H^\omega$ is essentially self-adjoint almost surely. 
\\\\
\noindent Here we will mostly focus on the spectrum of $H^\omega$ below zero. 
Since $-\Delta$ is non-negative operator, the negative spectrum of $H^\omega$ is associated with the negative part of the potential $V^\omega_{-}(x)=\min\{V^\omega(x), 0\}$.
Note that if the support of $\mu$ have lower bound than the negative spectrum is mainly determined by $\{a_n\}_n$, 
hence we are working with $\mu$ which satisfies the relation \eqref{probdisteq1}.
If we assume that the second moment exists for $\mu$ and $a_n$ decays fast enough at infinity ($\alpha>1$), 
then Cook's method \cite[Theorem XI.16]{sca-b} will ensure the existence of wave operators, which implies the absolutely continuous spectrum on $[0,\infty)$.
\\\\
\noindent There are many works which have dealt with this question, for example \cite{FGKM,KKO,Kr,B,FBB,SN}. 
The results presented in this work are not too surprising considering similar result do exists in discrete case, for example \cite{KKO,Kr}.
Figotin-Germinet-Klein-M\"{u}ller \cite{FGKM} considered the similar model with compactly supported $\mu$ with  absolutely continuous single site potential.
The authors showed that the random operator exhibits localization of eigenfunctions almost surely and dynamical localization below zero.
For $0<\alpha<2$ they showed that the operator $H^{\omega}$ has infinitely many negative eigenvalues almost surely and provided a bound 
for $N^{\omega}(E):=\#\{x\in\sigma(H^\omega):x<E\}$ in terms of stationary model (i.e choosing $a_n=1$).
\\\\
\noindent Kirsch-Krishna-Obermeit \cite{KKO} (see also \cite{Kr, MK} and \cite{KIRSCH}) considered $H^{\omega}=-\Delta+V^{\omega}$ on $\ell^2(\mathbb{Z}^d)$, 
where $V^{\omega}(n)=a_n\omega_n$ and $\Delta$ is the adjacency operator for the graph $\ZZ^d$. 
Under certain restriction on $\{a_n\}_n$ and $\mu$ they showed that $\sigma(H^{\omega})=\mathbb{R}$ and $\sigma_c(H^{\omega})\subseteq[-2d,2d]$ almost surely.
Jak{\v{s}}i{\'c}-Last \cite[Theorem 1.2]{JL} showed that singular spectrum in absent in $(-2d,2d)$ for $d\ge3$ and $a_n\sim |n|^{-\alpha}$ for some $\alpha>1$.
Hundertmark-Kirsch \cite{kirsch-hundermat} studied several models of $d$-dimensional Schr\"{o}dinger operators with non-stationary random potentials including  sparse potential. 
For a large class of sparse potential the authors establish the existence of 
absolutely continuous spectrum above zero.
Krutikov \cite{denis} studied the random Schr\"{o}dinger operator $H^\omega$, with bounded random sparse potential on $L^2(\mathbb{R}^d)$ for $d\ge 3$ 
and showed that $[0,\infty)\subset \sigma_{ac}(H^\omega)$ almost surely. The author used Cook's method to show the existence of 
wave operator.
Safronov \cite{saf} considered the above model with an extra disorder strength parameter for $d\geq 3$ and showed that
the absolutely continuous spectrum of $H^\omega$ is supported on $[0,\infty)$ almost surely for almost every realization of disorder strength,
under the condition that $\sum_{n\neq 0}\frac{{a_n}^2}{|n|^{d-1}}<\infty$ and $\{\omega_n\}_n$ are bounded random variables with zero mean.
Another interesting work is Frank-Safronov \cite{Fr-Sa}, which deals with the case of $V^\omega=\sum_{j\in J}\omega_j f_j(x)$,
where $\omega_j$ are bounded i.i.d random variables with finite variance and zero mean.
\\\\
For the case of growing potential, Gordon-Jaksic-Molchanov-Simon\cite{GJMS} considered the operator 
$$H^\omega=-\Delta+\displaystyle\sum_{n\in\mathbb{Z}^d}|n|^\alpha\omega_n|\delta_n\rangle\langle\delta_n|,\qquad\alpha>0 $$
on $\ell^2(\mathbb{Z}^d)$, where $\{\omega_n\}_n$ are i.i.d random variable with uniformly distributed on $[0,1]$. 
The authors showed that  the spectrum is discrete for $\alpha>d$, 
and for $\alpha \leq d$ there is a dense point spectrum along with discrete spectrum.
They also showed that the eigenfunctions are decaying faster than exponentially decay.
\\\\
\noindent In one dimension a larger class of operator could be studied in terms of Jacobi matrices. 
Jacobi matrices with decaying random potentials were first studied by Simon \cite{B} and Delyon-Simon-Souillard \cite{FBB}.
Kotani-Ushiroya \cite{SN} provided similar results in the case of Schr\"{o}dinger operators in one dimensional setting.
One dimensional case exhibits rich structure for the spectrum, hence we refer to work by Kiselev \cite{Kis}, 
Last-Simon \cite{LS} and Kiselev-Last-Simon \cite{KLS} and the literature therein for further references.
\\\\
\end{section}
We show that the random operator $H^\omega$ defined as in (\ref{model}) is essential self adjoint almost surely
for some restriction on $\alpha$ and $\delta$ (the Lemma \ref{lemma2} in Appendix).
For rest of the work we will be working under the following hypothesis.
\begin{hyp}\label{hyp}
~
\begin{enumerate}
\item[(1)] The single site distribution $\mu$ is absolutely continuous and has fat tail in the following sense
 $$\frac{d\mu}{dx}(x)=O\big(|x|^{-(1+\delta)}\big)~as~|x|\to\infty,~\delta>0.$$
\item[(2)] $\{a_n\}_{n\in\ZZ^d}$ satisfies $a_n= O(|n|^{-\alpha})$ as $|n|\rightarrow \infty$, for some $\alpha>0$ fixed. The norm $|n|$ will be used to denote $\|n\|_\infty$.
\item[(3)]  The parameters $\alpha$ and $\delta$ satisfies $(2+\alpha)\delta>d$.
\end{enumerate}
\end{hyp}
\begin{rem}
Any operator of the form \eqref{model} which satisfies all the above hypothesis are essentially self adjoint. It's proof can be found in Lemma \ref{lemma2}.
\end{rem}

\noindent We assume $\mu$ is absolutely continuous with bounded density to obtain good Wegner estimate which is uniform with respect to
center of the box as our potential is not ergodic. Here we will use modified version
of \cite[Lemma (Wegner estimate)]{FGKM} and this Wegner estimate is enough to get Initial scale estimate.
We refer Combes-Hislop-Klopp\cite{JHF}, Combes-Hislop \cite{CH}, Hundertmark-Killip-Nakamura-Stollmann-Veseli\'{c}
\cite{HKNSV}, Kirsch \cite{wer}, Kirsch-Veselic \cite{KIV} and Stollmann \cite{peter-stol} for more about Wegner estimate.  
To perform multiscale analysis in the region $(-\infty, 0)$, we need the Wegner estimate for small enough interval around $E$ for any $E\in(-\infty, 0)$. 
Third condition is there to guarantee that there is unique self adjoint extension of the densely defined operator $H^\omega$.
\\\\
Now we are ready to state our main result:
\begin{thm}
\label{main-thm} 
On $L^2(\RR^d)$ consider the operator $H^\omega$ defined by \eqref{model}, where $\mu$ and $\{a_n\}_{n\in\ZZ^d}$ satisfies the Hypothesis \ref{hyp}. Then
\begin{enumerate}
 \item[(a)] For $\alpha\delta\leq d$ we have $\sigma_{ess}(H^\omega)=\mathbb{R}$ almost surely.
 \item[(b)] For $\alpha\delta>d$, we have $\sigma_{ess}(H^\omega)=[0,\infty)$ and below zero the operator $H^\omega$ has discrete spectrum almost surely.
 \item[(c)] The operator $H^\omega$ shows exponential localization on $(-\infty, 0)$ almost surely.
 \item[(d)] For $\delta>2$ and $\alpha>1$ we have $[0,\infty)\subset\sigma_{ac}(H^\omega)$ almost surely.
\end{enumerate}
\end{thm}
\begin{rem}
~
\begin{enumerate}
\item Note that for $\delta>\frac{1}{2}$, there is a non-trivial set of $\alpha$ such that the condition (3) of hypothesis \ref{hyp} and condition of (a) in the above theorem are satisfied. 
 So the result of (a) in above theorem is not vacuous. 
\item On a similar note, for any $\delta>0$ we can choose $\alpha$ to be large enough so that condition (3) of hypothesis \ref{hyp} and condition of (b) in above theorem are satisfied.
\item For part (b) in the above theorem, the possibility ``zero is the only limit point for negative eigenvalues'' is not ruled out. 
But if we take  $(\alpha-2)\delta>d$, then Lemma \ref{lemma3} shows that there is only finitely many negative eigenvalues of $H^\omega$ almost surely.
\item The statement (c) imply that $\sigma_{c}(H^\omega)\subseteq [0,\infty)$ almost surely.
\item In (d) we took $\delta>2$ to ensure the existence of second moment of $\mu$. 
For Cook's method to show the existence of wave operator, we need $\alpha>1$.

\item Combining the statements of (d) and (a) of the theorem for $d\geq 3$, implies that 
$\sigma(H^\omega)=\mathbb{R}$ and $\sigma_c(H^\omega)=[0,\infty)$ almost surely, for some choice of $\alpha$ and $\delta$.
\end{enumerate}
\end{rem}
\begin{rem}
As a digression, it should be noted that we can define the random potential to be of the form
$$V^\omega(x)=Q(x)\sum_{n\in\ZZ^d}\omega_n\chi_{_{(0,1]^d}}(x-n),$$
where $Q:\RR^d\rightarrow [0,\infty)$ is continuous with 
$$\inf_{x\in (0,1]^d}Q(x+n)\sim |n|^{-\alpha}~~\&~~\sup_{x\in (0,1]^d}Q(x+n)\sim |n|^{-\alpha}$$
for $n$ large. This will not change the proofs fundamentally. 
\end{rem}

\noindent The statements of the theorem has two distinct parts. First two statements deal with the essential spectrum of the operator and second part deals with the nature of the spectrum,
i.e whether a point in the spectrum is in pure point part or in the absolutely continuous part.
For the negative part of the spectrum, we show exponential localization using the multiscale analysis.
This was first developed to show Anderson localization for Anderson tight binding model by Fr\"{o}hlich-Spencer \cite{FRS} and Fr\"{o}hlich-Martinelli-Spencer-Scoppolla \cite{FrMSS}
and later simplified by Von Dreifus \cite{Dr} and Von Dreifus-Klein \cite{DrK}.
Later the method was refined by Germinet-Klein \cite{GA} (See also \cite{AK}) and is known as Bootstrap multiscale analysis.
There are other methods of proving pure point spectrum, one of which is by showing dynamical localization.
Another method is through fractional moment method first introduced by Aizenman-Molchanov \cite{AM} (see also \cite{A,ASFH}) for discrete case 
and later improved to continuum model by Aizenman-Elgart-Naboko-Schenker-Stolz  \cite{AENSS}.
Though we have used the multiscale analysis here, fractional moment method is also applicable.
\\\\
\noindent For showing essential spectrum, we constructed Weyl sequence of approximate eigenfunctions for the operator $H^\omega$. 
This method works because the potential behaves like sparse potential and so we can find a sequence of boxes of increasing width such that except for one point everywhere the potential is negligible.
So, we can use a Weyl sequence of Laplacian itself to approximate the eigenfunctions for the full operator.
As for the statement involving discrete spectrum, we show that the eigenvalue counting function $N^\omega(E):=\#\{x\in\sigma(H^\omega):x<E\}$ is finite almost surely, for $E<0$ .
\\\\
\noindent In the rest of the article we use $V^\omega(n)$ to denote the potential at the box $n+(0,1]^d$ for $n\in\mathbb{Z}^d$. 
From our assumption on $a_n$ and $\mu$, we can work with $V^\omega(n)=\frac{\omega_n}{|n|^\alpha}$ for large enough $|n|$ and $V^\omega(0)=\omega_0$.
\begin{section}{Proof of the results}
The proof involving the essential spectrum in part (a) and (b) of the theorem  follows through the following lemma. 
The proof of the lemma is based on the work Kirsch-Krishna-Obermeit \cite{KKO}.
Part (a) of our main theorem is immediate and for part (b), the statement about essential spectrum also follows from the lemma.
\begin{lemma}\label{lemma1}
 If $H^\omega$ as in (\ref{model}) with Hypothesis \ref{hyp} then 
 \begin{enumerate}
 \item $\displaystyle\bigcup_{\lambda\in\mathbb{R}}\sigma (-\Delta+\lambda\chi_{[0,1)^d})=\mathbb{R}$.
 \item $[0,\infty)\subseteq\sigma_{ess}(H^\omega)$ almost surely.
  \item If $\alpha\delta\leq d$, then  $\displaystyle\bigcup_{\lambda\in\mathbb{R}}\sigma (-\Delta+\lambda\chi_{[0,1)^d})\subset\sigma_{ess}(H^\omega)$ almost surely.
 \end{enumerate}
 \end{lemma}
\begin{proof}~\\
{\bf 1:} First note that 
$$[0,\infty)\subset \sigma(-\Delta+\lambda \chi_{(0,1]^d})\qquad\forall \lambda\in\RR,$$ 
so we always have
$$\displaystyle [0,\infty)\subset \bigcup_{\lambda\in\mathbb{R}}\sigma(-\Delta+\lambda\chi_{[0,1)^d}).$$
Hence we only need to show  
$$(-\infty,0)\subset \bigcup_{\lambda\in\RR}\sigma(-\Delta+\lambda\chi_{[0,1)^d}).$$
Using the fact that the operator $-\Delta+\lambda\chi_{[0,1)^d}$  has finitely many negative eigenvalues for $\lambda<0$,
we can use the Hellmann-Feynman theorem to get
$$\frac{dE_\lambda^{min}}{d\lambda}=\dprod{\phi_\lambda}{\chi_{[0,1)^d}\phi_\lambda}\geq 0,$$
where $E_\lambda^{min}$ is the minimum eigenvalue of $-\Delta+\lambda\chi_{[0,1)^d}$ and $\phi_\lambda$ is the corresponding eigenfunction. 
So $E_\lambda^{min}$ is continuous monotonic function of $\lambda$. Now using the min-max principle, we obtain
$$E_{\lambda}^{min}\leq \dprod{f}{(-\Delta+\lambda \chi_{[0,1)^d})f}=\int_{\RR^d} |\nabla f|^2dx+\lambda\int_{[0,1)^d} |f|^2dx$$
for any $f\in C^\infty_c(\mathbb{R}^d)$. So fixing $f$, we get that $\displaystyle\lim_{\lambda\rightarrow-\infty}E^{min}_\lambda=-\infty$. 

Next using the fact that 
$$\lambda\chi_{[0,1)^d}\leq -\Delta+\lambda \chi_{[0,1)^d}\qquad \forall \lambda<0,$$
we have $\lambda<E^{min}_\lambda<0$ and so $\displaystyle\lim_{\lambda\uparrow 0}E^{min}_\lambda=0$. So using continuity of $E^{min}_\lambda$ we get the desired result.
\\\\ 
{\bf 2:} Define the event
\begin{equation}
 \label{posi-evnt}
 B^{\epsilon, k}_m=\{\omega: |V^\omega(j)|<\epsilon~~for~~|j-m|<k\},
\end{equation}
and note that for $|m|\gg k$, we have 
$$\mathbb{P}\big(B^{\epsilon, k}_m\big)\ge \bigg(1-\frac{1}{\epsilon^\delta|m|^{\alpha\delta}}\bigg)^{k^d}.$$ 
So for large enough $L$ we have $\mathbb{P}\big(B^{\epsilon, k}_m\big)>e^{-1}$ for $|m|>L$, hence
$$\sum_{m\in\ZZ^d}\mathbb{P}\big(B^{\epsilon, k}_m\big)=\infty.$$
Defining
\begin{align*}
\Omega_{n,\epsilon}&=\{\omega:\exists \{m_k\}_{k\in\mathbb{N}}\subset\mathbb{Z}^d~such~that~|m_p-m_q|>2n~\forall p\neq q\\
&\qquad\qquad\qquad\qquad~and~V^\omega(j)<\epsilon,~for~|j-m_k|<n ~\forall k\in\mathbb{N}\},
\end{align*}
and using previous observation along with Borel-Cantelli lemma, we conclude $\mathbb{P}(\Omega_{n,\epsilon})=1$ for any $n\in\mathbb{N}$ and $\epsilon>0$, 
so defining $\tilde{\Omega}=\cap_{n\ge 1}\Omega_{n, \frac{1}{n}}$, we have
\begin{equation}
 \mathbb{P}\big(\tilde{\Omega}\big)=1.
\end{equation}
Now for $E\in[0,\infty)$ observe that $\phi(x)=e^{\iota \sum_{i=1}^d k_ix_i}$ formally satisfies 
$$(-\Delta \phi)(x)=E\phi(x)\qquad\forall x\in\RR^d$$
for $\sum_i k_i^2=E$. Let $f\in C_c^\infty(\RR^d)$ be a non-negative function supported in the unit ball with
$$\int f(x)dx=1~\&~\left|\frac{\partial f}{\partial x_i}(x)\right|<M,~\left|\frac{\partial^2 f}{\partial x_i^2}(x)\right|<M\qquad \forall x\in\RR^d,1\leq i\leq d,$$
and define $\phi_r(x)=r^{-\frac{d}{2}}\phi(x)f(\frac{x}{r})$. Then observe that
\begin{align*}
\norm{(-\Delta-E)\phi_r}_2^2&=\int_{\RR^d}\left| \sum_{i=1}^d \left(\frac{2}{r^{\frac{d}{2}+1}}\frac{\partial \phi}{\partial x_i}(x)\frac{\partial f}{\partial x_i}\left(\frac{x}{r}\right)+ \frac{1}{r^{\frac{d}{2}+2}}\phi(x)\frac{\partial^2 f}{\partial x_i^2}\left(\frac{x}{r}\right)\right)\right|^2dx\\
&= O\left(\frac{1}{r^2}\right)\xrightarrow{r\rightarrow\infty}0
\end{align*}
By definition of  $\tilde{\Omega}$, for any $\omega\in\tilde{\Omega}$ there exists a sequence $\{x_n\}_{n\in\mathbb{N}}$ such that $|x_n-x_m|> n+m+1$ for all $n\neq m$ and
$$|V^\omega(j)|<\frac{1}{n}\qquad\forall |j-x_n|<n\qquad\forall n\in\mathbb{N}.$$
So defining $\psi_n(x)=\phi_{n}(x-x_n)$, we get $supp(\phi_n)\cap supp(\phi_m)=\phi$ for $n\neq m$ and
\begin{align*}
\big\|(H^\omega-E)\phi_n \big\|_2\leq  \big\|(-\Delta-E)\phi_n\big\|_2 +\big\|V^\omega\phi_n\big\|_2=O\left(\frac{1}{n}\right)\xrightarrow{n\rightarrow\infty}0.
\end{align*}
So $E$ is in the essential spectrum of $H^\omega$. Hence $[0,\infty)\subseteq\sigma_{ess}(H^\omega)$ almost surely.
\\\\
\noindent{\bf 3:}
Since we have already proven (2), we only need to show $\bigcup_{\lambda<0}\sigma(-\Delta+\lambda\chi_{(0,1]^d})\subset\sigma_{ess}(H^\omega)$ almost surely. 
So defining the events
$$A^{\epsilon, k}_m=\{\omega: V^\omega(m)\in(\lambda-\epsilon, \lambda+\epsilon)~and~|V^\omega(j)|<\epsilon~for~0<|j-m|<k\},$$
for $\lambda<0$ and $\epsilon>0$ such that $\lambda+\epsilon<0$,  observe that
\begin{align}\label{lem1eq1}
&\mathbb{P}\big(A^{\epsilon, k}_m| V^\omega(m)\in(\lambda-\epsilon, \lambda+\epsilon)\big)\nonumber \\
&\qquad\ge \prod_{0<|j-m|<k} \bigg(1-\frac{C}{\epsilon^\delta|j|^{\alpha\delta}}\bigg)\nonumber\\
&\qquad\approx \left(1-\frac{C}{\epsilon^\delta|m|^{\alpha\delta}}\right)^{k^d},
\end{align}
for $|m|\gg k$.
If $\alpha\delta\leq d$, set $0<\theta<\frac{\alpha \delta}{d}$ and $k_m=|m|^{\theta}$, then there exists $L$ such that for $m>L$
$$\left(1-\frac{C}{\epsilon^\delta|m|^{\alpha\delta}}\right)^{k_m^d}=e^{|m|^{d\theta}\ln\left(1-\frac{C}{\epsilon^\delta |m|^{\alpha\delta}}\right)}\approx e^{-C\frac{|m|^{d\theta-\alpha\delta}}{\epsilon^\delta}}\geq \frac{1}{2},$$
hence
\begin{align*}
&\sum_{m\in\mathbb{Z}^d}\mathbb{P}\big(A^{\epsilon, k_m}_m\big)=\sum_{|m|<L}\mathbb{P}\big(A^{\epsilon, k_m}_m\big)+\sum_{|m|>L}\mathbb{P}\big(A^{\epsilon, k_m}_m\big)\\
&\qquad \geq\sum_{|m|>L} \mathbb{P}[V^\omega(m)\in(\lambda-\epsilon,\lambda+\epsilon)]\left(1-\frac{C}{\epsilon^\delta|m|^{\alpha\delta}}\right)^{k^d_m}\\
&\qquad \ge \frac{1}{2}\sum_{|m|>L} \mathbb{P}[V^\omega(m)\in(\lambda-\epsilon,\lambda+\epsilon)]= \frac{C_\lambda (\epsilon)}{2}\sum_{|m|>L} \frac{1}{|m|^{\alpha\delta}}=\infty.
\end{align*}
In the above we used
$$\mathbb{P}[V^\omega(m)\in(\lambda-\epsilon,\lambda+\epsilon)] = \frac{C_\lambda(\epsilon)}{|m|^{\alpha\delta}},$$
which follows from (1) of Hypothesis \ref{hyp}. 
One can express $C_\lambda (\epsilon)\approx\frac{C}{|\lambda+\epsilon|^\delta} -\frac{C}{|\lambda-\epsilon|^\delta}$ for some $C>0$.

Now following argument similar to previous part and using Borel-Cantelli lemma, we have $\mathbb{P}\big(\Omega_{\lambda,\epsilon}\big)=1$, where
\begin{align*}
\Omega_{\lambda,\epsilon}&=\{\omega:\exists\{m_n\}_{n\in\mathbb{N}}\text{ in $\ZZ^d$ such that} |m_p-m_q|> k_{m_p}+k_{m_q}~\forall p,q\in\mathbb{N},~\&\\
&\qquad\qquad |V^\omega(m_p)-\lambda|<\epsilon,~|V^\omega(j)|<\epsilon~\forall 0<|j-m_p|<k_{m_p}~\forall p,q\in\mathbb{N} \}.
\end{align*}
Observe that for any $\omega\in\Omega_{\lambda,\epsilon}$, 
there exists an infinite sequence of disjoint cubes $\{\Lambda_{k_{m_p}}(m_p)\}_{p\in\mathbb{N}}$ such that $|V^\omega(m_p)-\lambda|<\epsilon$ and 
$|V^\omega(j)|<\epsilon$ for $j\in\Lambda_{k_{m_p}}(m_p)\setminus\{m_p\}$. 
Defining $\Omega_\lambda=\bigcap_{n=1}^\infty\Omega_{\lambda,\frac{1}{n}}$, we have $\mathbb{P}(\Omega_\lambda)=1$.
Now let $\{q_n\}_{n\in\mathbb{N}}$ be an enumeration of rationals and define $\tilde{\Omega}=\cap_{n\in\mathbb{N}} \Omega_{q_n}$, then we have $\mathbb{P}(\tilde{\Omega})=1$.

For some rational $\lambda<0$, let $E\in \sigma(-\Delta+\lambda\chi_{[0,1)^d})$, then we can construct a sequence of functions $\{\psi_n\}$ in $C_c^\infty(\RR^d)$ with $\big \|\psi_n\big \|_2=1$ such that
\begin{equation}
\label{ess}
\big \|(-\Delta-\lambda\chi_{[0,1)^d}-E)\psi_n\big \|_2<\frac{1}{n}.
\end{equation}
Since $\psi_n$ are compactly supported, let $\{r_n\}_{n\in\mathbb{N}}$ be such that $supp(\psi_n)\subset\Lambda_{r_n}(0)$. 
For any $\omega\in\tilde{\Omega}$, by definition there exists $\{m_n\}_{n\in\mathbb{N}}$ in $\ZZ^d$ such that 
$$|V^\omega(m_n)-\lambda|<\frac{1}{n}~~~\&~|V^\omega(j)|<\frac{1}{n}~\forall j\in\Lambda_{k_{m_n}}(m_n)\setminus\{m_n\},$$
and the cubes $\{\Lambda_{k_{m_n}}(m_n)\}_{n\in\mathbb{N}}$ are pairwise disjoint. 
Using the subsequence $\{m_{n_p}\}_{p\in\mathbb{N}}$ of $\{m_n\}_{n\in\mathbb{N}}$ which satisfies $k_{m_{n_p}}>r_p$, define $\phi_p(x)=\psi_p(x-m_{n_p})$ and observe that $\{\phi_p\}_{p\in\mathbb{N}}$ has disjoint support and 
\begin{align*}
\big\| (H^\omega-E)\phi_p\big\|_2&\leq \big\| (-\Delta+\lambda \chi_{m_{n_p}+[0,1)^d}-E)\phi_p\big\|_2\\
&\qquad+\norm{\left(-\lambda\chi_{m_{n_p}+[0,1)^d}+ \sum_{j\in\Lambda_{k_{m_{n_p}}}(m_{n_p})} V^\omega(j)\chi_{j+[0,1)^d}\right)\phi_p}_2\\
&\leq \frac{1}{p}+\frac{1}{n_p}\xrightarrow{p\rightarrow\infty}0.
\end{align*}
So for $E$, we have a sequence of approximate eigenvectors which are orthonormal, hence $E$ is in the essential spectrum of $H^\omega$ almost surely.

\end{proof}

\noindent As stated before, the results involving essential spectrum for the Theorem \ref{main-thm} follows from above lemma.
For part (b), we estimate the number of eigenvalues of $H^\omega$ below $E$, for $E<0$, and show that it is finite almost surely for $\alpha\delta>d$.
To show the exponential localization inside $(-\infty, 0)$, we use multiscale analysis.  
We use Combes-Thomas estimate \cite[Thorem 2.4.1]{stolman} together with Wegner estimate as given in \cite[Theorem (Wegner estimate)]{peter-stol} 
to get the required decay for Green's functions. 
For part (d) we verified the conditions given in \cite[Theorem 2.3]{denis} to show the wave operators exist.
\\\\
\noindent {\bf Proof of (a) for Theorem \ref{main-thm}:}\\
The conclusion of part (a) of the Theorem \ref{main-thm} follows from first and third part of the above lemma.
\\
\qed\\
{\bf Proof of (b) for Theorem \ref{main-thm}:}\\
For part (b) of Theorem \ref{main-thm}, we only need to show that the spectrum of $H^\omega$ below zero is discrete. 
This is because part (2) of lemma \ref{lemma1} already tells us that $[0,\infty)$ is in the essential spectrum.

Let $H^\omega_{n, N}$ denote the restriction of $H^\omega$ onto $n+(0,1]^d$ with Neumann boundary condition.
Let $\mathcal{N}^\omega(E)$ denote the number of eigenvalues of $H^\omega$ which are below $E$ and $\mathcal{N}^\omega_{n, N}(E)$ denote the same for $H^\omega_{n, N}$.
Using the fact that
\begin{equation}\label{dirih-neu}
\displaystyle\oplus_{n\in\mathbb{Z}^d}H^\omega_{n,N}\leq H^\omega.
\end{equation}
and Min-Max principle we have
\begin{equation}\label{eig-inq}
\mathcal{N}^\omega(E)\leq \sum_{n\in\mathbb{Z}^d}\mathcal{N}^\omega_{n, N}(E).
\end{equation}
Using \cite[Proposition 2]{RS} for $\epsilon>0$ we have
\begin{equation}\label{eig-esst}
\mathcal{N}^\omega_{n, N}(-\epsilon)\leq C_0\bigg(-\frac{\omega_n}{|n|^\alpha}-\epsilon\bigg)^{\frac{d}{2}}+
C_1\bigg[1+\bigg(-\frac{\omega_n}{|n|^\alpha}-\epsilon\bigg)^{\frac{d-1}{2}}\bigg],
\end{equation}
whenever $\frac{\omega_n}{|n|^\alpha}<-\epsilon$. For $\frac{\omega_n}{|n|^\alpha}>-\epsilon$, the operator  $H^\omega_{n, N}+\epsilon$ 
is positive definite, hence $\mathcal{N}^\omega_{n, N}(-\epsilon)=0$.\\
Let $\alpha\delta>d$, then using $\mu\big((-\infty,-R)\big)<\frac{C}{R^\delta}$ for large $R>0$ (follows from part (1) of Hypothesis \ref{hyp}), we have
$$\sum_{|n|>M}\mathbb{P}\bigg(\omega:\frac{\omega_n}{|n|^\alpha}<-\epsilon\bigg)<\frac{C}{\epsilon^\delta}\sum_{|n|>M}\frac{1}{|n|^{\alpha\delta}}<\infty $$
for $M\gg 1$.
Now Borel-Cantelli lemma will imply that almost surely  
$$\#\{n\in\ZZ^d: \mathcal{N}^\omega_{n, N}(-\epsilon)>0\}<\infty,$$
hence using \eqref{eig-inq} we have 
$$\mathcal{N}^\omega(-\epsilon)<\infty~~almost~surely. $$
Using (2) of lemma \ref{lemma1} we conclude the part (b) of the Theorem \ref{main-thm}.
\\
\qed\\
\noindent{\bf Proof of (c) for the Theorem \ref{main-thm}:}\\
To show exponential localization below zero, we will perform Multiscale analysis on $(-\infty, E')$ for each $E'<0$.
Here we use Bootstrap multiscale analysis \cite[Theorem 3.4]{GA} developed by Germinet-Klein.
To use multiscale analysis, we need two estimates, first Wegner estimate uniform with respect to center of the box, since our potential is not ergodic.
The second estimate is known as Initial scale estimate.
\\
But first let us setup few notations. Let $\Lambda_L(x)$  denote the cube centered at $x$ with side length $L$.
The restriction $H^\omega_{\Lambda_L(0)}$ of $H^\omega$ onto the Hilbert subspace $L^2(\Lambda_L(0))$ can be written as
\begin{align}
\label{res-box}
 H^\omega_{\Lambda_L(0)} &=-\Delta_L+\sum_{n\in\Lambda_L(0)}a_n\omega_n\chi_{_{(0,1]^d}}(x-n),
 \end{align}
 here we use Dirichlet boundary condition to define the Laplacian.
 For simplicity of notation we denote
 $$V^\omega(n)=a_n\omega_n~~and~~\Omega_L=\big\{\omega: |V^\omega(n)|<L^a,~n\in\Lambda_L(0) \big\},~for~some~a>0. $$
 ~\\\\
 {\bf Wegner Estimate:}\\
Here we use the following version of the Wegner estimate, 
\begin{equation}
 \label{weg}
\sup_{n\in\mathbb{Z}^d}\mathbb{P}\bigg(dist\big(\sigma\big(H^\omega_{\Lambda_L(n)}\big), E\big)
 <\eta \bigg\vert \Omega_L\bigg)\leq Q_s  \eta^s L^{d+\gamma a},~~E<E'<0.
\end{equation}
which is obtained by simple modification of \cite[Lemma (Wegner estimate)]{FGKM}.
In the estimate \eqref{weg}, the terms $\eta^s$ (for $0<s\leq 1$) and $L^d$ appears by same reasoning as  \cite[Lemma (Wegner estimate)]{FGKM}. 
The extra term $L^{\gamma a}$ (for some suitably chosen $\gamma \in\mathbb{N}$) shows up because 
 $\big|V^\omega(n)\big|\leq L^a~~\forall~n\in\Omega_L$. The constant $Q_s$ is uniform in $\alpha\ge0$ and $E\in (-\infty, E')$, once $E'<0$ is fixed.
 \\\\
{\bf Initial Scale estimate: }\\
Once we have the Wegner estimate (\ref{weg}), then using Combes-Thomas estimate \cite[Theorem 2.4.1]{stolman},
along with the fact that $dist\big(\partial\Lambda_L,\Lambda_{\frac{L}{3}}(0)\big) =O(L)$, we obtain the required decay for Green's function as follows:
\begin{align}
\label{ISE}
&\mathbb{P}\bigg(\bigg\|\chi_{_{\partial\Lambda_L}}\big(H^\omega_{_{\Lambda_L(0)}}-E\big)^{-1}\chi_{_{\Lambda_\frac{L}{3}(0)}}\bigg\|\leq ce^{-mL}\bigg)\\
&\qquad\qquad\qquad\ge \mathbb{P}\bigg(dist\big(\sigma\big(H^\omega_{\Lambda_L(0)}, E\big)>L^{-b}\bigg\vert\Omega_L\bigg)\mathbb{P}\big(\Omega_L\big)\nonumber\\
&\qquad\qquad\qquad\ge \bigg(1-\frac{Q_s}{L^{bs-d-\gamma a}}\bigg)\bigg(1-\frac{L^d}{L^{\delta(a+\alpha)}}\bigg)\nonumber\\
&\qquad\qquad\qquad=1-O\bigg(\frac{1}{L^{\min\{bs-d-\gamma a, a\delta+\alpha\delta-d\}}}\bigg).
\end{align}
Now if we choose $b>\frac{d+\gamma a}{s}$ and $a>\frac{d}{\delta}-\alpha$,  then
using \cite[Theorem 3.4]{GA} (or \cite[Theorem 5.4]{AK}), we can show that $(-\infty, 0)\subset\sigma_{pp}(H^\omega)$
with exponentially decaying eigenfunctions almost surely. \\
\qed
\\\\
\noindent{\bf Proof of (d) for the Theorem \ref{main-thm}:}\\
Here we use the modification of Cook's method as given in \cite[Theorem 2.3]{denis} to show the existence of wave operator. 
So if we take $\delta>2$ and $\alpha>1$
it is easy to verify the following two conditions:
\begin{equation}
 \int_{\mathbb{R}^d} \big(1+|x|\big)^{-2m}\big(V^\omega(x)\big)^2dx<\infty,~almost~ surely,~for~some~m>0.\nonumber
\end{equation}
\begin{equation}
 \int_1^\infty\bigg(\int_{a<|x|<b}\big(V^\omega(xt)\big)^2dx\bigg)^{\frac{1}{2}}dt<\infty,~almost~surely,~~0<a<b.\nonumber
\end{equation}
which completes the proof of the theorem.

\end{section}

\appendix\section{Appendix}
\begin{lemma}\label{lemma2}
For $(2+\alpha)\delta>d$, the operator $H^\omega$ is essentially self adjoint where the domain of definition is $C_c^\infty(\RR^d)$.
\end{lemma}
\begin{proof}
For $\omega\in\Omega$, define
$$V^\omega_{-}(x)=\sum_{\substack{n\in\mathbb{Z}^d}} |V^\omega(n)|~ \chi_{\{\omega_n<0\}}(\omega) ~\chi_{n+[0,1)^d}(x),$$
and observe that $V^\omega(x)+V^\omega_{-}(x)\geq 0$ for all $x\in\RR^d$ and almost all $\omega$. 
The essential self adjointness will follow from \cite[Theorem A]{EEM} once we have $V^\omega_{-}(x)=o(x^2)$ as $|x|\to\infty$ almost surely. 
Observe that for $0<\epsilon<2+\alpha-\frac{d}{\delta}$ and large enough $N$ we have
\begin{align*}
& \mathbb{P}[\omega\in\Omega: \exists |n|>N ~such~that~\omega_n<0~\&~|V^\omega(n)|>|n|^{2-\epsilon}]\\
&\qquad\leq \sum_{|n|>N}\mathbb{P}[|V^\omega(n)|>|n|^{2-\epsilon}~\&~\omega_n<0]\\
&\qquad\leq C_\delta\sum_{|n|>N} \frac{1}{|n|^{(2-\epsilon+\alpha)\delta}}\leq \tilde{C}_{d,\delta}\sum_{k\geq N} \frac{k^{d-1}}{k^{(2-\epsilon+\alpha)\delta}}\\
&\qquad\leq \frac{C^\prime_{d,\delta}}{N^{(2-\epsilon+\alpha)\delta-d}},
\end{align*}
which converges to zero under the condition $(2+\alpha-\epsilon)\delta>d$. Hence using  Borel-Cantelli lemma we have
$$\mathbb{P}\left[\bigcap_{N\in\mathbb{N}} \bigcup_{n=N}^\infty \{\omega: \exists |m|=n ~such~that~\omega_m<0~\&~|V^\omega(m)|>n^{2-\epsilon} \}\right]=0$$
Hence defining
\begin{equation*}
\tilde{\Omega}:=\left\{\omega: \#\{n\in\mathbb{Z}^d: \omega_n<0~\&~|V^\omega(n)|>|n|^{2-\epsilon}\}<\infty\right\},
\end{equation*}
we have $\mathbb{P}\big(\tilde{\Omega}\big)=1$.
So the random variable
$$N^\omega=\max\{|n|:\omega_n<0~\&~|V^\omega(n)|>|n|^{2-\epsilon}\},$$
is finite  almost surely, hence setting 
$$M^\omega=\max\left\{2,\max\left\{ \frac{|V^\omega(n)|}{(1+|n|)^{2-\epsilon}}: \omega_n<0~\&~|n|\leq N^\omega\right\}\right\},$$ 
we observe that
$$V^\omega_{-}(x)\leq M^\omega\sum_{n\in\ZZ^d} (1+|n|)^{2-\epsilon} \chi_{n+[0,1)^d}(x)\leq  M^\omega (1+|x|)^{2-\epsilon}.$$
So using \cite[Theorem A]{EEM}, we conclude that the operator $H^\omega$ is essentially self adjoint on the domain $C_c^\infty(\RR^d)$ almost surely.
\end{proof}

\noindent Next lemma shows that if $\mu$ and $\{a_n\}_n$ decays fast enough at infinity, then there are finitely many negative eigenvalues of $H^\omega$ almost surely.
\begin{lemma}\label{lemma3}
For $(\alpha-2)\delta>d$, the operator $H^\omega$ defined by \eqref{model} has finitely many negative eigenvalues
\end{lemma}
\begin{proof}
For $0<\epsilon<\alpha-\frac{d}{\delta}$ define
$$A_N=\{\omega: \exists |n|>N~such~that~\omega_n<0~\&~|V^\omega(n)|>|n|^{-\epsilon}\}.$$
Then observe that
\begin{align*}
\mathbb{P}(A_N)&\leq \sum_{|n|>N}\mathbb{P}(\omega_n<0~\&~|V^\omega(n)|>|n|^{-\epsilon})\\
&\leq \sum_{|n|>N}\frac{C}{|n|^{(\alpha-\epsilon)\delta}}\approx C\sum_{k>N}\frac{k^{d-1}}{k^{(\alpha-\epsilon)\delta}}=O\left(\frac{1}{N^{(\alpha-\epsilon)\delta-d}}\right),
\end{align*}
so like lemma \ref{lemma2}, we can use Borel-Cantelli lemma to prove that
\begin{equation*}
\mathbb{P}\left(\omega:\text{ $\omega_n<0~\&~|V^\omega(n)|>|n|^{-\epsilon}$ only for finitely many $n$ in $\ZZ^d$} \right)=1,
\end{equation*}
which in turn implies
$$N^\omega:=\max\{|n|:\omega_n<0~\&~|V^\omega(n)|>|n|^{-\epsilon}\}$$
is finite almost surely. Hence defining 
$$M^\omega=\max\left\{2,\max\left\{\frac{|V^\omega(n)|}{1+|n|^\epsilon}:|n|\leq N^\omega\right\}\right\},$$
we see that 
\begin{equation}\label{lem3eq1}
 H^\omega\geq -\Delta-\frac{M^\omega}{1+|x|^\epsilon}\qquad almost~surely.
\end{equation}
So if $\alpha-\frac{d}{\delta}>2$, then we can choose $\epsilon>2$ and since RHS of \eqref{lem3eq1} has finitely many negative eigenvalues, we get the desired result.

\end{proof}



\begin{thebibliography}{99}
 \bibitem{AK} Klein, Abel: \textsl{Multiscale analysis and localization of random operators, Random
Schr¨odinger operators}, Panor. Synth\`{e}ses, vol. 25, Soc. Math. France, Paris, 2008,
pp. 121-159.
 
\bibitem{GA} Germinet, Francois; Klein, Abel: \textsl{Bootstrap multiscale analysis and localization in random media}. 
Comm. Math. Phys. {\bf 222} (2001), 415-448.

\bibitem{CH}Combes, J-M; Hislop, P. D: \textsl{Localization for some continuous, random Hamiltonians in d-dimensions}. J. Funct. Anal. {\bf 124(1)},
(1994), 149-180.

\bibitem{RS} Reed, Michael; Simon, Barry: \textsl{Method of Modern Mathematical Physics {\bf I}
  Analysis of Operators}, Academic Press, 1978.
  
 \bibitem{stolman}  Stollmann, Peter: \textsl {Caught by disorder. 
 Bound states in random media}. Progress in Mathematical Physics, {\bf 20}. Birkh\"{a}user Boston, Inc,
 Boston, MA, 2001.
 
 \bibitem{peter-stol} Stollmann: \textsl{Peter Wegner estimates and localization for continuum Anderson models with some singular distributions}. 
 Arch. Math. (Basel), {\bf 75(4)}, (2000), 307-311.
 
 \bibitem{bourgain} Bourgain, J: \textsl{On random Schr\"{o}dinger operators on $\mathbb{Z}^2$} , Discrete Contin. Dyn.Syst. {\bf 8}, 1-15, (2002).

  \bibitem{FGKM} Figotin, Alexander; Germinet, Francois; Klein, Abel; M\"{u}ller, Peter:
\textsl{Persistence of Anderson localization in Schrödinger operators with decaying random potentials}.  Ark. Mat. {\bf 45(1)}, (2007), 15-30.

\bibitem{KKO} Kirsch, W; Krishna M; Obermeit, J: \textsl{ Anderson model with decaying randomness: mobility edge}. Math. Z. {\bf 235(3)} (2000), 421-433.

\bibitem{wer} Kirsch, W: \textsl{Wegner estimates and Anderson localization for alloy-type potentials}. Math. Z. {\bf 221}, (1996) 507-512.
\bibitem{KIV} Kirsch, W; Veselic I, \textsl{Wegner estimate for sparse and other generalized alloy type potentials}. Proc. Ind. Acad. {\bf 112(1)} (2002),131-146.

\bibitem{KIRSCH} Kirsch, W: \textsl{Scattering theory for sparse random potentials}, Random Oper. Stochastic
Equations {\bf 10(4)}, 329–334, (2002).

\bibitem{kirsch-hundermat}  Hundertmark, Dirk; Kirsch, Werner:
\textsl{Spectral theory of sparse potentials}, Stochastic processes, physics and geometry: new interplays, 
I (Leipzig, 1999), CMS Conf. Proc., 28, Amer. Math. Soc., Providence, RI, (2000),  213-238.

\bibitem{Kr} Krishna, M: \textsl {Anderson model with decaying randomness existence of extended
states}. Proc. Indian Acad. Sci. Math. Sci. {\bf 100}, (1990), 285-294.

\bibitem{MK} Krishna, M.: \textsl{Absolutely continuous spectrum and spectral transition for some continuous random operators}. 
Proc. Indian Acad. Sci. Math. Sci. 122(2),  243-255, (2012). 

\bibitem{HKNSV} D. Hundertmark; R. Killip; S. Nakamura; P. Stollmann; I. Veseli\'{c}:
\textsl{Bounds on the spectral shift function and the density of states},. Comm. Math.
Phys. {\bf 262}, (2006), 489-503.


\bibitem{JL} Jak\v{s}i\'{s}, Vojkan; Last, Yoram: \textsl{Spectral structure of Anderson type Hamiltonians}. Invent.Math {\bf 141(3)}, (2000), 561-577.

\bibitem{EEM} Eastham, M. S. P., Evans, W. D.,  McLeod, J. B. . \textsl{The essential self-adjointness of Schr\"{o}dinger-type operators}. Archive for Rational Mechanics and Analysis, {\bf 60(2)},(1976), 185-204.

\bibitem{FRS} Fr\"{o}hlich, J; Spencer, T: \textsl{Absence of diffusion with Anderson tight binding model
or large disorder or low energy}. Commun. Math. Phys. {\bf 88}, (1983), 151-184.

\bibitem{FrMSS} Fr\"{o}hlich, J; Martinelli, F; Scoppola, E; Spencer, T: \textsl{Constructive proof of localization in the Anderson tight binding model}.
Commun. Math. Phys. {\bf 101}, 21-46, 1985.

\bibitem{Dr} Von Dreifus, H: {On the effects of randomness in ferromagnetic models and Schr\"{o}dinger Ph.D}. thesis, New York University (1987).

\bibitem{DrK} Von Dreifus, H; Klein, A: \textsl{A new proof of localization in the Anderson tight
binding model}. Commun. Math. Phys. 124, 285-299 (1989).



\bibitem{A} Aizenman, M.: {Localization at weak disorder: some elementary bounds. Rev.
Math. Phys}. {\bf 6}, 1163-1182 (1994).

\bibitem{AM} Aizenman, M; Molchanov, S:\textsl{Localization at large disorder and extreme energies:
an elementary derivation}. Commun. Math. Phys. {\bf 157}, 245-278, (1993).

\bibitem{ASFH} Aizenman, M; Schenker, J; Friedrich, R; Hundertmark, D: \textsl{Finite volume
fractional-moment criteria for Anderson localization}. Commun. Math. Phys. {\bf 224},
219-253, (2001).

\bibitem{AENSS} Aizenman, M; Elgart, A; Naboko, S; Schenker, J; Stolz, G: \textsl{Moment analysis
for localization in random Schr\"{o}dinger operators}.  Invent. Math. {\bf 163(1)}, 343-413, (2006).

\bibitem{B} B, Simon: {Some Jacobi matrices with decaying potentials and dense point spectrum}.
Comm. Math. Phys. {\bf 87}, 253-258, (1982).

\bibitem{FBB} F. Delyon; B. Simon; B. Souillard: \textsl{From power pure point to continuous spectrum
in disordered systems}. Ann. Institute H. Poincar\'{e}. {\bf 42}, 283-309, (1985).

\bibitem{SN} S. Kotani; N. Ushiroya: {One-dimensional Schr\"{o}dinger operators with random
decaying potentials}. Commun. Math. Phys. {\bf 115}, 247-266, (1988).

\bibitem{denis} Krutikov, Denis: \textsl{Schr\"{o}dinger operators with random sparse potentials. Existence of wave operators}. 
Lett. Math. Phys. {\bf67(2)}, 133-139, (2004).

\bibitem{Fr-Sa} Frank, R. L; Safronov, O: \textsl{Absolutely continuous spectrum of a class of random nonergodic Schrödinger operators}. 
Int.Math.Res.Not. {\bf(42)}, 2559-2577, (2005).

\bibitem{saf} Safronov, O: {Absolutely continuous spectrum of a one-parameter family of Schr\"{o}dinger operators}.
Algebra i Analiz {\bf 24(6)}, 178-195, (2012).

\bibitem{BCH}  Barbaroux, J. M; Combes, J. M; Hislop, P. D: \textsl{Localization near band edges for random Schr\"{o}dinger operators}.
Helv. Phys. Acta {\bf 70(1-2)}, 16-43, (1997).

\bibitem{Kis} A. Kiselev: {Absolutely continuous spectrum of one dimensional Schr\"{o}dinger operators
and Jacobi matrices with slowly decreasing potentials}. Commun. Math. Phys. {\bf 179}
377-400, (1996).

\bibitem{sca-b} Reed, M; Simon, B:  \textsl{Methods of Modern Mathematical Physics, III. Scattering Theory},
Academic Press, London, (1979).

\bibitem{LS} Y. Last; B. Simon: \textsl{Eigenfunctions, transfer matrices and absolutely continuous
spectrum of one dimensional Schr\"{o}dinger operators}. Invent.math. {\bf135}, 329-367, (1999).

\bibitem{KLS} A. Kiselev: Y. Last: B: Simon. \textsl{Modified Pr\"{u}fer and EFGP transforms and the
pectral analysis of one-dimensional Schr\"{o}dinger operators}. Commun. Math. Phys. {\bf 194}, 1-45, (1998).

\bibitem{JHF} Jean-Combes; Peter, Hislop; Fr\'{e}d\'{e}ric, Klopp: \textsl{An optimal Wegner estimate and its application to the global continuity
of the Integrated density of states for random Schr\"{o}dinger operators}. Duke Math. J. {\bf 140(3)}, 469-498, (2007).

\bibitem{GJMS} Gordon, Y. A; Jak\v{s}i\'{c}, V; Molchanov, S; Simon, B:
\textsl{Spectral properties of random Schrödinger operators with unbounded potentials}. Comm. Math. Phys. {157(1)}, 23-50, (1993). 

















\end{thebibliography}
\end{document}